\documentclass[11pt,a4paper]{article}

\usepackage{inputenc}
\usepackage{amsmath}
\usepackage{bm}
\usepackage{bbold}
\usepackage{amsthm}
\usepackage{enumerate}

\usepackage{hyperref}

\title{Tropical optimization problems in\\ time-constrained project scheduling\thanks{This work was supported in part by the Russian Foundation for Humanities (grant No. 16-02-00059).}}

\author{N. Krivulin\thanks{Faculty of Mathematics and Mechanics, St.~Petersburg State University, 28 Universitetsky Ave., St.~Petersburg, 198504, Russia, 
nkk@math.spbu.ru.}
}

\date{}

\newtheorem{theorem}{Theorem}
\newtheorem{lemma}[theorem]{Lemma}

\theoremstyle{definition}
\newtheorem{example}{Example}

\begin{document}

\maketitle

\begin{abstract}
We consider a project that consists of activities to be performed in parallel under various temporal constraints, which include start-start, start-finish and finish-start precedence relationships, release times, deadlines, and due dates. Scheduling problems are formulated to find optimal schedules for the project with respect to different objective functions to be minimized, such as the project makespan, the maximum deviation from the due dates, the maximum flow-time, and the maximum deviation of finish times. We represent these problems as optimization problems in terms of tropical mathematics, and then solve them by applying direct solution methods of tropical optimization. As a result, new direct solutions of the scheduling problems are obtained in a compact vector form, which is ready for further analysis and practical implementation. The solutions are illustrated by simple numerical examples.
\\

\\

\end{abstract}

\section{Introduction}

Tropical optimization problems, which are formulated and solved in the framework of tropical mathematics, find increasing use in various fields of operations research, including project scheduling. As an applied mathematical discipline concentrated on the theory and applications of semirings with idempotent addition, tropical mathematics dates back to a few seminal papers \cite{Pandit1961Anew,Cuninghamegreen1962Describing,Giffler1963Scheduling,Hoffman1963Onabstract,Vorobjev1963Theextremal,Romanovskii1964Asymptotic,Korbut1965Extremal}, including those \cite{Cuninghamegreen1962Describing,Giffler1963Scheduling} concerned with optimization problems drawn from machine scheduling.

In succeeding years, tropical mathematics was studied by many authors under various names, such as idempotent algebra, max-algebra, max-plus algebra, extremal algebra (see \cite{Baccelli1993Synchronization,Kolokoltsov1997Idempotent,Golan2003Semirings,Heidergott2006Maxplus,Akian2007Maxplus,Litvinov2007Themaslov,Gondran2008Graphs,Speyer2009Tropical,Butkovic2010Maxlinear} and references therein). Tropical optimization problems were investigated in a number of works, in which scheduling issues frequently helped to motivate and illustrate the study. Specifically, tropical mathematics (max-plus algebra) served as a solution framework for scheduling problems in \cite{Cuninghamegreen1979Minimax,Zimmermann1981Linear,Zimmermann1984Some,Baccelli1993Synchronization,Zimmermann2003Disjunctive,Bouquard2006Application,Fiedler2006Linear,Heidergott2006Maxplus,Butkovic2009Onsome,Butkovic2010Maxlinear,Houssin2011Cyclic,Aminu2012Nonlinear}.

Many optimization problems are formulated in the tropical mathematics setting to minimize nonlinear functions defined on vectors over idempotent semifields (semirings with multiplicative inverses), subject to constraints given by vector equations and inequalities (see, e.g., overviews in \cite{Krivulin2014Tropical,Krivulin2015Amultidimensional}). For some problems, direct solutions are obtained in a closed form under general assumptions. Other problems are solved algorithmically by using iterative computational procedures.

This paper examines problems that are drawn from time-constrained project scheduling, which involves the planning of activities in a project over time to achieve certain objectives \cite{Demeulemeester2002Project,Neumann2003Project,Tkindt2006Multicriteria,Vanhoucke2012Project}. The aim of the study is to provide new solutions to the problems on the basis of recent results in tropical optimization.

We consider scheduling problems, which are to find an optimal schedule for a project that consists of a set of activities operating in parallel under various temporal constraints, including start-start, start-finish, finish-start, release time, deadline, and due-date constraints. As optimization criteria to minimize, we take the project makespan, the maximum deviation from due dates, the maximum flow-time, and the maximum deviation of finish times. Such problems are known to have algorithmic solutions in the form of iterative computational procedures. Specifically, many problems can be formulated and solved as linear, integer or mixed-integer linear programs by appropriate mathematical programming algorithms. Examples of the algorithmic solutions can be found in \cite{Pritsker1969Azeroone,Dereyck1998Abranchandbound,Artigues2000Apolynomial} (see also reviews in \cite{Demeulemeester2002Project,Neumann2003Project,Tkindt2006Multicriteria,Vanhoucke2012Project}).

We represent the scheduling problems as tropical optimization problems, which are then solved by applying solution methods developed in \cite{Krivulin2014Aconstrained,Krivulin2015Amultidimensional,Krivulin2015Extremal}. We derive new direct solutions to the scheduling problems considered, which, in contrast to the conventional algorithmic solutions, provide results in a compact explicit vector form, ready for further analysis and applications, and thus have the potential to complement and supplement existing approaches. These solutions allow various constraints to be incorporated in a unified and constructive way. The calculation of the solutions involves simple matrix and vector computations according to explicit formulae, which offers a basis for the development of efficient computational algorithms and software implementation.

The current paper further extends and improve the results presented in the conference paper \cite{Krivulin2015Tropicaloptimization} in an effort to incorporate illustrative numerical examples, discuss computational complexity of solutions, enhance formulae, and refine the reference list. In the paper, we continue the research on the application of tropical optimization to scheduling problems reported in journal and conference publications \cite{Krivulin2012Algebraic,Krivulin2014Aconstrained,Krivulin2015Amultidimensional,Krivulin2015Extremal,Krivulin2015Tropical}. We examine more complicated general scheduling problems with new objective functions introduced and additional temporal constraints imposed. We show how tropical optimization techniques can be applied to these problems, which results in new solutions in both tropical optimization and project scheduling. 

The rest of the paper is organized as follows. In Section~\ref{S-PSMEP}, we describe scheduling problems that motivate and illustrate the study, and formulate these problems by using the conventional  notation. Section~\ref{S-PADR} includes a brief overview of preliminary definitions and results of tropical mathematics to be used in the subsequent sections. In Section~\ref{S-TOP}, we present some tropical optimization problems and their solutions, and discuss the computational complexity of the solutions. In Section~\ref{S-APT}, we first rewrite the scheduling problems as tropical optimization problems, and then solve them by applying the results from Section~\ref{S-TOP}. The solutions are illustrated by simple, but representative, numerical examples.

\section{Project scheduling model and example problems}
\label{S-PSMEP}

We start with the description of a project scheduling model in a general form, and then present example problems to find optimal schedules. To facilitate the subsequent representation of the problems in terms of tropical mathematics, we employ a somewhat different notation than that commonly adopted in the literature (see, e.g., \cite{Demeulemeester2002Project,Neumann2003Project,Tkindt2006Multicriteria,Vanhoucke2012Project} for further details and standard notations of project scheduling). 

Consider a project that consists of $n$ activities operating in parallel under start-start, start-finish and finish-start precedence relations, due dates, and time boundaries for start and finish times in the form of release time, release deadline and deadline constraints. To describe the temporal constraints and scheduling objectives under consideration, we use the symbols $x_{i}$ and $y_{i}$, which, respectively, represent the unknown start and finish times for each activity $i=1,\ldots,n$.

\subsection{Temporal constraints}

We now examine constraints imposed on the start and finish times of each activity $i=1,\ldots,n$. First, we represent precedence relations, which link activity $i$ with other activities. Let $a_{ij}$ be the minimum possible time lag between the start of activity $j$ and the finish of $i$. The time lag $a_{ii}$ specifies the minimum duration of the activity (the duration provided that no other constraints are imposed). Note that the value $-a_{ij}$ can be interpreted as the maximum time lag between the finish of $j$ and the start of $i$. If there is no lag defined, we assume $a_{ij}=-\infty$.

The start-finish constraints take the form of the inequalities $a_{ij}+x_{j}\leq y_{i}$ holding for all $j=1,\ldots,n$. The activity is assumed to finish immediately after all related start-finish constraints are satisfied, and thus at least one of the inequalities must hold as an equality. Then, these inequalities combine to give one inequality, which, under the above assumption of immediate finish, is equivalent to the equality
\begin{equation*}
\max_{1\leq j\leq n}(a_{ij}+x_{j})
=
y_{i}.
\end{equation*}

Furthermore, we use the notation $b_{ij}$ to describe the minimum time lag between the start of activity $j$ and the start of $i$, and put $b_{ij}=-\infty$ if the lag is not specified. The start-start constraints are given by the inequalities $b_{ij}+x_{j}\leq x_{i}$ for all $j$, which can readily be rewritten as one inequality
\begin{equation*}
\max_{1\leq j\leq n}(b_{ij}+x_{j})
\leq
x_{i}.
\end{equation*}

Let the minimum time lag between the finish of activity $j$ and the start of $i$ be denoted by $c_{ij}$, with $c_{ij}=-\infty$ if undefined. The finish-start constraints are written as the inequalities $c_{ij}+y_{j}\leq x_{i}$ for all $j$, or as one inequality
\begin{equation*}
\max_{1\leq j\leq n}(c_{ij}+y_{j})
\leq
x_{i}.
\end{equation*}

Finally, we introduce due dates and time boundary constraints. The due date indicates the time when the activity is ideally expected to finish. Since the due date may be unachievable under other constraints, it is considered as not a strict constraint. For activity $i$, we denote the due date by $d_{i}$.

Let $g_{i}$ and $h_{i}$ be the earliest and latest possible times to start, and $f_{i}$ be the latest possible time to finish. The release time, release deadline, and deadline constraints provide strict lower and upper boundaries for the start and finish times, given by 
\begin{equation*}
g_{i}
\leq
x_{i}
\leq
h_{i},
\qquad
y_{i}
\leq
f_{i}.
\end{equation*}

\subsection{Optimization criteria}

To describe scheduling objectives, we use several criteria that commonly arise in the development of optimal schedules in practice. The criteria are written below in the form, which is ready for immediate translation into terms of tropical mathematics. 

We begin with the maximum absolute deviation of finish times of activities from due dates that a project should meet. The minimum value of this criterion corresponds to the least violation of the due dates, which can be attained. With the notation introduced above, the maximum deviation from the due dates is given by
\begin{equation*}
\max_{1\leq i\leq n}|y_{i}-d_{i}|
=
\max_{1\leq i\leq n}\max(y_{i}-d_{i},d_{i}-y_{i}).
\end{equation*}

Next, we consider the maximum deviation of completion times of all activities. The minimization of this criterion is equivalent to finding a schedule, where all activities have to finish simultaneously as much as possible. Such a problem can arise in just-in-time manufacturing, when certain delivery operations must be completed at once. The maximum deviation of completion times is written as 
\begin{equation*}
\max_{1\leq i\leq n}y_{i}
-
\min_{1\leq i\leq n}y_{i}
=
\max_{1\leq i\leq n}y_{i}
+
\max_{1\leq i\leq n}(-y_{i}).
\end{equation*}

The flow-time of an activity (also known as the system, throughput and turn-around time) is defined as the difference between its start and finish times, and can determine expenses related to undertaking the activity in a project. The flow-time of activity $i$ is bounded from below by the value of $a_{ii}$, which is commonly assumed to be nonnegative, and may be greater than $a_{ii}$ due to other temporal constraints.

In many real-world problems, the objective is formulated to minimize the maximum flow-time taken over all activities, and thus described by the expression
\begin{equation*}
\max_{1\leq i\leq n}(y_{i}-x_{i}).
\end{equation*}

Finally, we discuss the makespan, which is the interval between the earliest start time and the latest finish time of activities in a project. The makespan indicates the total duration of the project, and finds wide application as an objective function to be minimized in many scheduling problems. The makespan is given by
\begin{equation*}
\max_{1\leq i\leq n}y_{i}
-
\min_{1\leq i\leq n}x_{i}
=
\max_{1\leq i\leq n}y_{i}
+
\max_{1\leq i\leq n}(-x_{i}).
\end{equation*}

\subsection{Examples of scheduling problems}

We conclude with typical examples of scheduling problems, which are to serve to both motivate and illustrate the results in the rest of the paper. To formulate the problems, we use the notation and formulae introduced above for the unknown variables, given parameters, temporal constraints and scheduling objectives.

\subsubsection{Minimization of maximum deviation from due dates}

First, we consider a problem to minimize the maximum deviation from due dates under start-finish, start-start and finish-start constraints. Given the parameters $a_{ij}$, $b_{ij}$, $c_{ij}$ and $d_{i}$, the problem is to find the unknown start time $x_{i}$ and finish time $y_{i}$ for each activity $i=1,\ldots,n$, that 
\begin{equation}
\begin{aligned}
&
\text{minimize}
&&
\max_{1\leq i\leq n}\max(y_{i}-d_{i},d_{i}-y_{i}),
\\
&
\text{subject to}
&&
\max_{1\leq j\leq n}(a_{ij}+x_{j})
=
y_{i},
\quad
\max_{1\leq j\leq n}(b_{ij}+x_{j})
\leq
x_{i},
\\
&&&
\max_{1\leq j\leq n}(c_{ij}+y_{j})
\leq
x_{i},
\quad
i=1,\ldots,n.
\end{aligned}
\label{P-yididiyi-aijxjyi-bijxjxi-cijyjxi}
\end{equation}

\subsubsection{Minimization of maximum deviation of finish times}

We now formulate a problem of minimizing the maximum deviation of finish times, subject to start-finish, start-start, finish-start, and deadline constraints. Given the parameters $a_{ij}$, $b_{ij}$, $c_{ij}$ and $f_{i}$, we need to determine the unknowns $x_{i}$ and $y_{i}$ that 
\begin{equation}
\begin{aligned}
&
\text{minimize}
&&
\max_{1\leq i\leq n}y_{i}
+
\max_{1\leq i\leq n}(-y_{i}),
\\
&
\text{subject to}
&&
\max_{1\leq j\leq n}(a_{ij}+x_{j})
=
y_{i},
\quad
\max_{1\leq j\leq n}(b_{ij}+x_{j})
\leq
x_{i},
\\
&&&
\max_{1\leq j\leq n}(c_{ij}+y_{j})
\leq
x_{i},
\quad
y_{i}
\leq
f_{i},
\qquad
i=1,\ldots,n.
\end{aligned}
\label{P-yiyi-aijxjyi-bijxjxj-cijyjxi-yifi}
\end{equation}

\subsubsection{Minimization of maximum flow-time}

Next, we consider a problem of minimizing the maximum flow-time under start-finish, start-start, finish-start and release time constraints. Given the parameters $a_{ij}$, $b_{ij}$, $c_{ij}$ and $g_{i}$, we find the values of $x_{i}$ and $y_{i}$ that solve the problem
\begin{equation}
\begin{aligned}
&
\text{minimize}
&&
\max_{1\leq i\leq n}(y_{i}-x_{i}),
\\
&
\text{subject to}
&&
\max_{1\leq j\leq n}(a_{ij}+x_{j})
=
y_{i},
\quad
\max_{1\leq j\leq n}(b_{ij}+x_{j})
\leq
x_{i},
\\
&&&
\max_{1\leq j\leq n}(c_{ij}+y_{j})
\leq
x_{i},
\quad
g_{i}
\leq
x_{i},
\qquad
i=1,\ldots,n.
\end{aligned}
\label{P-yixi-aijxjyi-bijxjxi-cijyjxi-gixi}
\end{equation}

\subsubsection{Minimization of makespan}

Suppose that we need to minimize the makespan of a project subject to start-finish, release time, release deadline, and deadline temporal constraints. Given $a_{ij}$, $g_{i}$, $h_{i}$ and $f_{i}$, we find $x_{i}$ and $y_{i}$ that
\begin{equation}
\begin{aligned}
&
\text{minimize}
&&
\max_{1\leq i\leq n}y_{i}
+
\max_{1\leq i\leq n}(-x_{i}),
\\
&
\text{subject to}
&&
\max_{1\leq j\leq n}(a_{ij}+x_{j})
=
y_{i},
\\
&&&
g_{i}
\leq
x_{i}
\leq
h_{i},
\quad
y_{i}
\leq
f_{i},
\qquad
i=1,\ldots,n.
\end{aligned}
\label{P-yixi-aijxjyi-gixihi-yifi}
\end{equation}

Note that problems like those presented above can normally be solved using iterative computational algorithms (see, e.g., \cite{Demeulemeester2002Project,Neumann2003Project,Tkindt2006Multicriteria,Vanhoucke2012Project} for overviews of available solutions). Specifically, these problems can be formulated as linear programs to solve them by computational methods of linear programming, which generally offer algorithmic solutions. Below, we provide new solutions to the problems, which are based on optimization methods in tropical mathematics, and present results in a compact explicit vector form rather than in the form of a numerical algorithm.

\section{Preliminary algebraic definitions and results}
\label{S-PADR}

In this section, we give a brief overview of preliminary definitions, notation and results of tropical algebra to provide a formal basis for the description and application of tropical optimization problems in the next sections. Both introductory and advanced material on tropical mathematics can be found in many publications, including \cite{Baccelli1993Synchronization,Kolokoltsov1997Idempotent,Golan2003Semirings,Heidergott2006Maxplus,Akian2007Maxplus,Litvinov2007Themaslov,Gondran2008Graphs,Speyer2009Tropical,Butkovic2010Maxlinear} to name only a few. The overview given below is mainly based on the presentation of results in \cite{Krivulin2014Aconstrained,Krivulin2015Amultidimensional,Krivulin2015Extremal}, which offers a useful framework to obtain direct solutions to the problems under study in a compact vector form.
 
\subsection{Idempotent semifield}

Consider a system $(\mathbb{X},\oplus,\otimes,\mathbb{0},\mathbb{1})$, where $\mathbb{X}$ is a set, which is closed under addition $\oplus$ and multiplication $\otimes$ with zero $\mathbb{0}$ and identity $\mathbb{1}$, such that $(\mathbb{X},\oplus,\mathbb{0})$ is a commutative idempotent monoid, $(\mathbb{X}\setminus\{\mathbb{0}\},\otimes,\mathbb{1})$ is an Abelian group, multiplication is  distributive over addition, and $\mathbb{0}$ is absorbing for multiplication. This system is usually called the idempotent semifield.

Addition is idempotent, which means that $x\oplus x=x$ for each $x\in\mathbb{X}$. The idempotent addition induces on $\mathbb{X}$ a partial order such that $x\leq y$ if and only if $x\oplus y=y$. It follows directly from the definition that $x\leq x\oplus y$ and $y\leq x\oplus y$ for all $x,y\in\mathbb{X}$. Moreover, the inequality $x\oplus y\leq z$ appears to be equivalent to the two inequalities $x\leq z$ and $y\leq z$, and both addition and multiplication are monotone in each argument. Finally, the partial order extends to a linear order on $\mathbb{X}$.
 
Multiplication is invertible to let each nonzero $x\in\mathbb{X}$ have the inverse $x^{-1}$ such that $x\otimes x^{-1}=\mathbb{1}$. The inverse operation is antitone, which implies that, for all nonzero $x$ and $y$, the inequality $x\leq y$ yields $x^{-1}\geq y^{-1}$.

The power notation with integer exponents is routinely used to represent repeated multiplication, and defined as follows: $x^{0}=\mathbb{1}$, $x^{p}=x\otimes x^{p-1}$ and $x^{-p}=(x^{-1})^{p}$ for all nonzero $x$ and positive integer $p$. The integer powers are assumed to extend to rational exponents to make $\mathbb{X}$ algebraically complete.

In the algebraic expressions below, the multiplication sign $\otimes$ is omited to save writing, and the exponents are read in the sense of tropical algebra. 

An example of the idempotent semifield under consideration is the real semifield $\mathbb{R}_{\max,+}=(\mathbb{R}\cup\{-\infty\},\max,+,-\infty,0)$, in which the addition $\oplus$ is defined as maximum, and the multiplication $\otimes$ is as ordinary addition, with the zero $\mathbb{0}$ given by $-\infty$, and the identity $\mathbb{1}$ by $0$. Each number $x\in\mathbb{R}$ has the inverse $x^{-1}$ equal to the opposite number $-x$ in the conventional notation. For all $x,y\in\mathbb{R}$, the power $x^{y}$ is well-defined and coincides with the arithmetic product $xy$. The partial order induced by idempotent addition corresponds to the standard linear order on $\mathbb{R}$.


\subsection{Matrices and vectors}

We now examine matrices and vectors over the idempotent semifield introduced above. The set of matrices that have $m$ rows and $n$ columns with entries from $\mathbb{X}$ is denoted $\mathbb{X}^{m\times n}$. A matrix with all entries equal to $\mathbb{0}$ is the zero matrix. A matrix is called row-regular (column-regular), if it has no rows (columns) that consist entirely of $\mathbb{0}$. Provided that a matrix is both row- and column-regular, it is regular.

For any matrices $\bm{A},\bm{B}\in\mathbb{X}^{m\times n}$ and $\bm{C}\in\mathbb{X}^{n\times l}$, and a scalar $x\in\mathbb{X}$, the matrix addition, matrix multiplication and scalar multiplication follow the standard rules with the scalar operations $\oplus$ and $\otimes$ in the place of the ordinary addition and multiplication, and are given by the formulae
\begin{gather*}
\{\bm{A}\oplus\bm{B}\}_{ij}
=
\{\bm{A}\}_{ij}\oplus\{\bm{B}\}_{ij},
\qquad
\{\bm{A}\bm{C}\}_{ij}
=
\bigoplus_{k=1}^{n}\{\bm{A}\}_{ik}\{\bm{C}\}_{kj},
\\
\{x\bm{A}\}_{ij}
=
x\{\bm{A}\}_{ij}.
\end{gather*}

The partial order associated with the idempotent addition and its properties extend to the matrices, where the relations are expanded entry-wise.

For any matrix $\bm{A}=(a_{ij})\in\mathbb{X}^{m\times n}$, the transpose is the matrix $\bm{A}^{T}\in\mathbb{X}^{n\times m}$.

The multiplicative conjugate transpose of $\bm{A}$ is the matrix $\bm{A}^{-}=(a_{ij}^{-})\in\mathbb{X}^{n\times m}$ with the entries $a_{ij}^{-}=a_{ji}^{-1}$ if $a_{ji}\ne\mathbb{0}$, and $a_{ij}^{-}=\mathbb{0}$ otherwise.

Consider square matrices of order $n$ in the set $\mathbb{X}^{n\times n}$. A matrix having the diagonal entries equal to $\mathbb{1}$, and the off-diagonal entries to $\mathbb{0}$, is the identity matrix denoted $\bm{I}$. The power notation with nonnegative integer exponents serves to represent iterated products as follows: $\bm{A}^{0}=\bm{I}$ and $\bm{A}^{p}=\bm{A}^{p-1}\bm{A}$ for any matrix $\bm{A}$ and integer $p>0$.

For any matrix $\bm{A}=(a_{ij})\in\mathbb{X}^{n\times n}$, the trace is given by
\begin{equation*}
\mathop\mathrm{tr}\bm{A}
=
a_{11}\oplus\cdots\oplus a_{nn}
=
\bigoplus_{i=1}^{n}a_{ii}.
\end{equation*}

For any matrices $\bm{A}$ and $\bm{B}$, and scalar $x$, the following identities are valid:
$$
\mathop\mathrm{tr}(\bm{A}\oplus\bm{B})
=
\mathop\mathrm{tr}\bm{A}
\oplus
\mathop\mathrm{tr}\bm{B},
\qquad
\mathop\mathrm{tr}(\bm{A}\bm{B})
=
\mathop\mathrm{tr}(\bm{B}\bm{A}),
\qquad
\mathop\mathrm{tr}(x\bm{A})
=
x\mathop\mathrm{tr}\bm{A}.
$$

Every matrix having only one row (column) is considered a row (column) vector. All vectors below are column vectors unless otherwise specified. The column vectors of order $n$ form the set $\mathbb{X}^{n}$. A vector with all elements equal to $\mathbb{0}$ is the zero vector. If a vector has no zero elements, it is regular. The vector of all ones is $\bm{1}=(\mathbb{1},\ldots,\mathbb{1})^{T}$.

Let $\bm{A}\in\mathbb{X}^{n\times n}$ be a row-regular matrix and $\bm{x}\in\mathbb{X}^{n}$ a regular vector. Then, the vector $\bm{A}\bm{x}$ is regular. If $\bm{A}$ is column-regular, then the row vector $\bm{x}^{T}\bm{A}$ is regular.

The multiplicative conjugate transpose of a nonzero column vector $\bm{x}=(x_{i})\in\mathbb{X}^{n}$ is a row vector $\bm{x}^{-}=(x_{i}^{-})$ with the entries $x_{i}^{-}=x_{i}^{-1}$ if $x_{i}\ne\mathbb{0}$, and $x_{i}=\mathbb{0}$ otherwise.

The conjugate transposition has the following useful properties. First, for any nonzero vector $\bm{x}$, the equality $\bm{x}^{-}\bm{x}=\mathbb{1}$ is valid.

Furthermore, if $\bm{x}$ and $\bm{y}$ are regular vectors of the same order, then the element-wise inequality $\bm{x}\leq\bm{y}$ is equivalent to $\bm{x}^{-}\geq\bm{y}^{-}$. In addition, the matrix inequality $\bm{x}\bm{y}^{-}\geq(\bm{x}^{-}\bm{y})^{-1}\bm{I}$ holds, and becomes the inequality $\bm{x}\bm{x}^{-}\geq\bm{I}$ when $\bm{y}=\bm{x}$.

Finally, consider a square matrix $\bm{A}\in\mathbb{X}^{n\times n}$. A scalar $\lambda\in\mathbb{X}$ is an eigenvalue of $\bm{A}$, if there exists a nonzero vector $\bm{x}\in\mathbb{X}^{n}$ to satisfy the equality $\bm{A}\bm{x}=\lambda\bm{x}$. The maximum eigenvalue of the matrix $\bm{A}$ is called the spectral radius, and given by
\begin{equation*}
\lambda
=
\mathop\mathrm{tr}\bm{A}
\oplus
\cdots
\oplus
\mathop\mathrm{tr}\nolimits^{1/n}(\bm{A}^{n})
=
\bigoplus_{k=1}^{n}\mathop\mathrm{tr}\nolimits^{1/k}(\bm{A}^{k}).
\end{equation*}

\subsection{Solution to linear inequalities}

We conclude the overview of the preliminary results with the solutions to vector inequalities, which arise in the literature in different settings and have solutions given by many authors in various forms (see, e.g., \cite{Cuninghamegreen1979Minimax,Baccelli1993Synchronization}). Below, we describe explicit solutions represented in a compact closed form that provides a unified framework for the systematic analysis of optimization problems in what follows.

Suppose that, given a matrix $\bm{A}\in\mathbb{X}^{m\times n}$ and a regular vector $\bm{d}\in\mathbb{X}^{m}$, we find vectors $\bm{x}\in\mathbb{X}^{n}$ to solve the inequality
\begin{equation}
\bm{A}\bm{x}
\leq
\bm{d}.
\label{I-Ax-d}
\end{equation}

A direct solution, which uses algebraic properties of the multiplicative conjugate transposition to reduce \eqref{I-Ax-d} to an equivalent inequality, is obtained in \cite{Krivulin2015Extremal} as follows.
\begin{lemma}
\label{L-Ax-d}
For any column-regular matrix $\bm{A}$ and regular vector $\bm{d}$, all solutions to inequality \eqref{I-Ax-d} are given by
\begin{equation*}
\bm{x}
\leq
(\bm{d}^{-}\bm{A})^{-}.
\label{I-xd-A}
\end{equation*}
\end{lemma}

Furthermore, we consider the problem: given a matrix $\bm{A}\in\mathbb{X}^{n\times n}$, find regular vectors $\bm{x}\in\mathbb{X}^{n}$ that satisfy the inequality
\begin{equation}
\bm{A}\bm{x}
\leq
\bm{x}.
\label{I-Ax-x}
\end{equation}

To represent a solution to the problem for any matrix $\bm{A}\in\mathbb{X}^{n\times n}$, we define a function that takes $\bm{A}$ to the scalar
\begin{equation*}
\mathop\mathrm{Tr}(\bm{A})
=
\mathop\mathrm{tr}\bm{A}\oplus\cdots\oplus\mathop\mathrm{tr}\bm{A}^{n}
=
\bigoplus_{k=1}^{n}
\mathop\mathrm{tr}\bm{A}^{k},
\end{equation*}
and make use of the asterisk operator (also known as the Kleene star), which maps $\bm{A}$ with $\mathop\mathrm{Tr}(\bm{A})\leq\mathbb{1}$ to the matrix
\begin{equation*}
\bm{A}^{\ast}
=
\bm{I}\oplus\bm{A}\oplus\cdots\oplus\bm{A}^{n-1}
=
\bigoplus_{k=0}^{n-1}\bm{A}^{k}.
\end{equation*}

The next result, based on the properties of the asterisk operator, is derived in \cite{Krivulin2015Extremal} to offer a complete solution to inequality \eqref{I-Ax-x} (see also \cite{Krivulin2006Solution,Krivulin2015Amultidimensional}).
\begin{theorem}
\label{T-Ax-x}
For any matrix $\bm{A}$, the following statements hold:
\begin{enumerate}
\item If $\mathop\mathrm{Tr}(\bm{A})\leq\mathbb{1}$, then all regular solutions to \eqref{I-Ax-x} are given by $\bm{x}=\bm{A}^{\ast}\bm{u}$, where $\bm{u}$ is any regular vector.
\item If $\mathop\mathrm{Tr}(\bm{A})>\mathbb{1}$, then there is no regular solution.
\end{enumerate}
\end{theorem}

The results, given by Lemma~\ref{L-Ax-d} and Theorem~\ref{T-Ax-x}, provide necessary instruments for the solutions of optimization problems presented in the next section, as well as for the application of the solutions to scheduling problems in the last section.

\section{Tropical optimization problems}
\label{S-TOP}

Tropical optimization problems present an area in tropical mathematics, which is of both theoretical interest and practical importance (see, e.g., \cite{Cuninghamegreen1979Minimax,Zimmermann1981Linear,Butkovic2010Maxlinear,Krivulin2014Tropical} for further details and overviews). Many problems are formulated in the framework of tropical mathematics to minimize nonlinear functions defined on vectors over idempotent semifields and calculated using multiplicative conjugate transposition of vectors. These problems may have constraints, which are given by vector equations and inequalities. There are problems that can be solved directly in a general setting. For other problems, only algorithmic solutions are known, which offer an iterative computational scheme to produce a solution, or signify that no solutions exist. 

The purpose of this section is twofold: first, to offer representative examples to demonstrate a variety of optimization problems under study, and second, to provide an efficient basis for the solution of scheduling problems in the next section. We consider examples of both unconstrained and constrained optimization problems with different objective functions defined in the common setting in terms of a general idempotent semifield. For all problems, direct solutions are given in a compact vector form ready for further analysis and straightforward computations. For some problems, the solutions obtained are complete solutions.

\subsection{Examples of optimization problems}

We begin with the following problem. Suppose that, given a matrix $\bm{A}\in\mathbb{X}^{m\times n}$ and a vector $\bm{d}\in\mathbb{X}^{m}$, we need to find regular vectors $\bm{x}\in\mathbb{X}^{n}$ that
\begin{equation}
\begin{aligned}
&
\text{minimize}
&&
\bm{d}^{-}\bm{A}\bm{x}\oplus(\bm{A}\bm{x})^{-}\bm{d}.
\end{aligned}
\label{P-dAxAxd}
\end{equation}

The next statement offers a direct algebraic solution to the problem, which is obtained by deriving a strict lower bound for the objective function together with a particular solution that yields the bound (see, e.g., \cite{Krivulin2012Asolution}).
\begin{theorem}
\label{T-dAxAxd}
Let $\bm{A}$ be a row-regular matrix and $\bm{d}$ a regular vector. Then, the minimum value in problem \eqref{P-dAxAxd} is equal to
\begin{equation*}
\Delta
=
((\bm{A}(\bm{d}^{-}\bm{A})^{-})^{-}\bm{d})^{1/2},
\end{equation*}
and the maximum solution is given by
\begin{equation*}
\bm{x}
=
\Delta(\bm{d}^{-}\bm{A})^{-}.
\end{equation*}
\end{theorem}

Furthermore, suppose that, given matrices $\bm{A},\bm{B}\in\mathbb{X}^{m\times n}$ and vectors $\bm{p},\bm{q}\in\mathbb{X}^{m}$, the purpose is to obtain regular vectors $\bm{x}\in\mathbb{X}^{n}$ to solve the problem
\begin{equation}
\begin{aligned}
&
\text{minimize}
&&
\bm{q}^{-}\bm{B}\bm{x}(\bm{A}\bm{x})^{-}\bm{p}.
\end{aligned}
\label{P-qBxAxp}
\end{equation}

A direct solution to the problem can be found, using a similar technique of evaluating a strict lower bound as above, in the following form \cite{Krivulin2013Explicit}.
\begin{theorem}
\label{T-qBxAxp}
Let $\bm{A}$ be row-regular and $\bm{B}$ column-regular matrices, $\bm{p}$ be nonzero and $\bm{q}$ regular vectors. Then, the minimum value in problem \eqref{P-qBxAxp} is equal to
\begin{equation*}
\Delta
=
(\bm{A}(\bm{q}^{-}\bm{B})^{-})^{-}\bm{p},
\end{equation*}
and attained at any vector
\begin{equation*}
\bm{x}
=
\alpha(\bm{q}^{-}\bm{B})^{-},
\qquad
\alpha>\mathbb{0}.
\end{equation*}
\end{theorem}

Given matrices $\bm{A},\bm{B}\in\mathbb{X}^{n\times n}$ and a vector $\bm{g}\in\mathbb{X}^{n}$, consider the problem to find regular vectors $\bm{x}\in\mathbb{X}^{n}$ that
\begin{equation}
\begin{aligned}
&
\text{minimize}
&&
\bm{x}^{-}\bm{A}\bm{x},
\\
&
\text{subject to}
&&
\bm{B}\bm{x}\oplus\bm{g}
\leq
\bm{x}.
\end{aligned}
\label{P-xAx-Bxgx}
\end{equation}

A direct solution to the problem is given in \cite{Krivulin2014Aconstrained,Krivulin2015Amultidimensional} by introducing an auxiliary parameter to represent the minimum value of the objective function, and reducing the problem to a parameterized vector inequality. The existence condition for solutions of the inequality is used to evaluate the parameter, whereas the solutions of the inequality are taken as the complete solution to the optimization problem. 
\begin{theorem}
\label{T-xAx-Bxgx}
Let $\bm{A}$ be a matrix with spectral radius $\lambda>\mathbb{0}$, and $\bm{B}$ be a matrix such that $\mathop\mathrm{Tr}(\bm{B})\leq\mathbb{1}$. Then, the minimum value in problem \eqref{P-xAx-Bxgx} is equal to
\begin{equation*}
\theta
=
\lambda
\oplus
\bigoplus_{k=1}^{n-1}\mathop{\bigoplus\hspace{1.0em}}_{1\leq i_{1}+\cdots+i_{k}\leq n-k}\mathop\mathrm{tr}\nolimits^{1/k}(\bm{A}\bm{B}^{i_{1}}\cdots\bm{A}\bm{B}^{i_{k}}),
\end{equation*}
and all regular solutions are given by
\begin{equation*}
\bm{x}
=
(\theta^{-1}\bm{A}\oplus\bm{B})^{\ast}\bm{u},
\qquad
\bm{u}
\geq
\bm{g}.
\end{equation*}
\end{theorem}

Finally, suppose that, given a matrix $\bm{A}\in\mathbb{X}^{n\times n}$ and vectors $\bm{g},\bm{h}\in\mathbb{X}^{n}$, we need to find regular vectors $\bm{x}\in\mathbb{X}^{n}$ that solve the problem
\begin{equation}
\begin{aligned}
&
\text{minimize}
&&
\bm{x}^{-}\bm{A}\bm{x},
\\
&
\text{subject to}
&&
\bm{g}
\leq
\bm{x}
\leq
\bm{h}.
\end{aligned}
\label{P-xAx-gxh}
\end{equation}

The same technique as above yields the next solution proposed in \cite{Krivulin2014Aconstrained}.
\begin{theorem}
\label{T-xAx-gxh}
Let $\bm{A}$ be a matrix with spectral radius $\lambda>\mathbb{0}$, and $\bm{h}$ be a regular vector such that $\bm{h}^{-}\bm{g}\leq\mathbb{1}$. Then, the minimum value in problem \eqref{P-xAx-gxh} is equal to
\begin{equation*}
\theta
=
\lambda
\oplus
\bigoplus_{k=1}^{n-1}(\bm{h}^{-}\bm{A}^{k}\bm{g})^{1/k},
\end{equation*}
and all regular solutions are given by
\begin{equation*}
\bm{x}
=
(\theta^{-1}\bm{A})^{\ast}\bm{u},
\qquad
\bm{g}
\leq
\bm{u}
\leq
(\bm{h}^{-}(\theta^{-1}\bm{A})^{\ast})^{-}.
\end{equation*}
\end{theorem}

\subsection{Computational complexity of solutions}

We conclude this section with some remarks on the computational complexity of the results with respect to the dimension $n$ of the problems. First, note that the computation of the solutions provided by Theorems~\ref{T-dAxAxd}, \ref{T-qBxAxp} and \ref{T-xAx-gxh} to problems \eqref{P-dAxAxd}, \eqref{P-qBxAxp} and \eqref{P-xAx-gxh} involves only a fixed number of basic matrix-vector operations, and thus obviously has a polynomial complexity. Specifically, the calculation of the star $\bm{A}^{\ast}$ for a matrix $\bm{A}$ of order $n$ requires computing of the sum of the $n-1$ first powers of $\bm{A}$, and therefore, takes no more than $O(n^{4})$ scalar operations. The computation of the spectral radius $\lambda$ for the matrix has the same level of complexity.

Consider the solution given by Theorem~\ref{T-xAx-Bxgx} to problem \eqref{P-xAx-Bxgx}, and verify that it can be calculated in polynomial time. We write the minimum in the problem as 
\begin{equation*}
\theta
=
\lambda
\oplus
\mu,
\qquad
\mu
=
\bigoplus_{k=1}^{n-1}\bigoplus_{l=1}^{n-k}\mathop\mathrm{tr}\nolimits^{1/k}(\bm{T}_{kl}),
\end{equation*}
where $\bm{T}_{kl}$ is the sum of matrix products $\bm{A}\bm{B}^{i_{1}}\cdots\bm{A}\bm{B}^{i_{k}}$ over all nonnegative integers $i_{1},\ldots,i_{k}$ such that $i_{1}+\cdots+i_{k}=l$.

The matrices $\bm{T}_{kl}$ satisfy the recurrence equation $\bm{T}_{kl}=\bm{T}_{k-1,l}\bm{A}\oplus\bm{T}_{k,l-1}\bm{B}$, where $\bm{T}_{k0}=\bm{A}^{k}$, $\bm{T}_{0l}=\bm{B}^{l}$ and $\bm{T}_{00}=\bm{I}$, which involves one matrix addition and two matrix multiplications per matrix. Since the number of matrices $\bm{T}_{kl}$ needed to evaluate the sum $\mu$ is $1+\cdots+n=n(n+1)/2$, the complexity of computing $\mu$, and hence of $\theta$, is at most $O(n^{5})$. Moreover, given the value of $\theta$, any solution vector $\bm{x}$ is computed in polynomial time as well, and thus we conclude that the overall solution has polynomial complexity.

\section{Application to project scheduling}
\label{S-APT}

We are now in a position to derive solutions to the scheduling problems formulated in the beginning of the paper. In this section, we first represent each problem in the framework of the idempotent semifield $\mathbb{R}_{\max,+}$ in both scalar and vector forms, and then solve it by reducing to an optimization problem of the previous section.

To provide insight into the solution method proposed, and to illustrate the computation technique used, we offer numerical examples of solving scheduling problems. To save room, we restrict the examples to artificial problems of low dimension, which, however, unambiguously demonstrate the practicability of the approach to solve real-world problems of high dimension.

\subsection{Minimization of maximum deviation from due dates}

Consider problem \eqref{P-yididiyi-aijxjyi-bijxjxi-cijyjxi} and describe it in terms of the semifield $\mathbb{R}_{\max,+}$. By replacing the usual operations by those of $\mathbb{R}_{\max,+}$, we obtain the problem to find the unknowns $x_{i}$ and $y_{i}$ for all $i=1,\ldots,n$ to
\begin{equation*}
\begin{aligned}
&
\text{minimize}
&&
\bigoplus_{i=1}^{n}(d_{i}^{-1}y_{i}\oplus y_{i}^{-1}d_{i}),
\\
&
\text{subject to}
&&
\bigoplus_{j=1}^{n}a_{ij}x_{j}
=
y_{i},
\quad
\bigoplus_{j=1}^{n}b_{ij}x_{j}
\leq
x_{i},
\quad
\bigoplus_{j=1}^{n}c_{ij}y_{j}
\leq
x_{i},
\\
&&&
i=1,\ldots,n.
\end{aligned}
\end{equation*}

To put the problem in a vector form, we introduce the matrix-vector notation
\begin{equation*}
\bm{A}
=
(a_{ij}),
\quad
\bm{B}
=
(b_{ij}),
\quad
\bm{C}
=
(c_{ij}),
\quad
\bm{d}
=
(d_{i}),
\quad
\bm{x}
=
(x_{i}),
\quad
\bm{y}
=
(y_{i}).
\end{equation*}

With this notation, the problem is to find the unknown vectors $\bm{x}$ and $\bm{y}$ that
\begin{equation}
\begin{aligned}
&
\text{minimize}
&&
\bm{d}^{-}\bm{y}\oplus\bm{y}^{-}\bm{d},
\\
&
\text{subject to}
&&
\bm{A}\bm{x}
=
\bm{y},
\quad
\bm{B}\bm{x}
\leq
\bm{x},
\quad
\bm{C}\bm{y}
\leq
\bm{x}.
\end{aligned}
\label{P-dyyd-Axy-Bxx-Cyx}
\end{equation}

The next result offers a solution to the problem.
\begin{theorem}
\label{T-dyyd-Axy-Bxx-Cyx}
Let $\bm{A}$ be a row-regular matrix, the matrix $\bm{D}=\bm{B}\oplus\bm{C}\bm{A}$ satisfy the condition $\mathop\mathrm{Tr}(\bm{D})\leq\mathbb{1}$, and $\bm{d}$ be a regular vector. Then, the minimum value in problem \eqref{P-dyyd-Axy-Bxx-Cyx} is equal to
\begin{equation*}
\Delta
=
((\bm{A}\bm{D}^{\ast}(\bm{d}^{-}\bm{A}\bm{D}^{\ast})^{-})^{-}\bm{d})^{1/2},
\end{equation*}
and the maximum solution is given by
\begin{equation*}
\bm{x}
=
\Delta\bm{D}^{\ast}(\bm{d}^{-}\bm{A}\bm{D}^{\ast})^{-},
\qquad
\bm{y}
=
\Delta\bm{A}\bm{D}^{\ast}(\bm{d}^{-}\bm{A}\bm{D}^{\ast})^{-}.
\end{equation*}
\end{theorem}
\begin{proof}
After substitution $\bm{y}=\bm{A}\bm{x}$, we combine both inequality constraints into one inequality of the form $\bm{D}\bm{x}\leq\bm{x}$ with $\bm{D}=\bm{B}\oplus\bm{C}\bm{A}$. An application of Theorem~\ref{T-Ax-x} to this inequality yields the solution $\bm{x}=\bm{D}^{\ast}\bm{u}$, where $\bm{u}$ is any regular vector.

Substitution of this solution reduces problem \eqref{P-dyyd-Axy-Bxx-Cyx} to the unconstrained problem
\begin{equation*}
\begin{aligned}
&
\text{minimize}
&&
\bm{d}^{-}\bm{A}\bm{D}^{\ast}\bm{u}\oplus(\bm{A}\bm{D}^{\ast}\bm{u})^{-}\bm{d}.
\end{aligned}
\end{equation*}

The last problem has the form of \eqref{P-dAxAxd} with $\bm{A}$ replaced by $\bm{A}\bm{D}^{\ast}$. Therefore, we can apply Theorem~\ref{T-dAxAxd} to obtain a solution in terms of the unknown vector $\bm{u}$. Turning back to the vectors $\bm{x}$ and $\bm{y}$ leads to the desired result.
\end{proof}

We now consider the conditions of the theorem to see that, in the context of project scheduling, these and similar conditions are naturally met in practice, unless the problem under consideration is incorrectly posed, due to wrong or incompatible conditions. Specifically, since the elements on the diagonal of the matrix $\bm{A}$ represent the minimum duration of activities, and thus must be greater than $\mathbb{0}=-\infty$, the matrix is formally both row- and column-regular. Moreover, for the same reason, the matrix $\bm{A}$ has the spectral radius $\lambda>\mathbb{0}$.

The regularity assumption on the vector $\bm{d}$ is equivalent to assuming that the deadlines for all activities are finite (not equal to $\mathbb{0}=-\infty$), which is the case for the real-world problems, and thus this assumption is typically fulfilled. The same conclusion can be reached for similar vectors encountered in the subsequent proofs.

Finally, the condition $\mathop\mathrm{Tr}(\bm{D})\leq\mathbb{1}$ implies that the constraints in the problem are compatible to provide nonempty feasible sets for the unknown vectors $\bm{x}$ and $\bm{y}$. 

We illustrate the direct solution offered by Theorem~\ref{T-dyyd-Axy-Bxx-Cyx} by the following example.

\begin{example}
\label{X-dyyd-Axy-Bxx-Cyx}
Consider a project that involves $n=3$ activities that operate under start-finish, start-start, finish-start and due dates constraints, given by the following matrices and the vector:
\begin{gather*}
\bm{A}
=
\left(
\begin{array}{crr}
4 & 0 & \mathbb{0}
\\
1 & 3 & -1
\\
0 & -2 & 2
\end{array}
\right),
\qquad
\bm{B}
=
\left(
\begin{array}{rrc}
\mathbb{0} & -2 & 1
\\
0 & \mathbb{0} & 2
\\
-1 & \mathbb{0} & \mathbb{0}
\end{array}
\right),
\\
\bm{C}
=
\left(
\begin{array}{ccr}
\mathbb{0} & \mathbb{0} & -1
\\
\mathbb{0} & \mathbb{0} & 1
\\
\mathbb{0} & \mathbb{0} & \mathbb{0}
\end{array}
\right),
\qquad
\bm{d}
=
\left(
\begin{array}{c}
5
\\
5
\\
5
\end{array}
\right),
\end{gather*}
where the symbol $\mathbb{0}=-\infty$ is used to simplify the writing of matrices.

To verify the conditions of Theorem~\ref{T-dyyd-Axy-Bxx-Cyx}, we first note that the matrix $\bm{A}$ is row-regular and the vector $\bm{d}$ is regular. Next, we calculate the matrices
\begin{equation*}
\bm{C}\bm{A}
=
\left(
\begin{array}{rrc}
-1 & -3 & 1
\\
1 & -1 & 3
\\
\mathbb{0} & \mathbb{0} & \mathbb{0}
\end{array}
\right),
\qquad
\bm{D}
=
\bm{B}\oplus\bm{C}\bm{A}
=
\left(
\begin{array}{rrc}
-1 & -2 & 1
\\
1 & -1 & 3
\\
-1 & \mathbb{0} & \mathbb{0}
\end{array}
\right),
\end{equation*}
and then take the matrix $\bm{D}$ to find the powers
\begin{equation*}
\bm{D}^{2}
=
\left(
\begin{array}{rrc}
0 & -3 & 1
\\
2 & -1 & 2
\\
-2 & -3 & 0
\end{array}
\right),
\qquad
\bm{D}^{3}
=
\left(
\begin{array}{rrc}
0 & -2 & 1
\\
1 & 0 & 3
\\
-1 & -4 & 0
\end{array}
\right).
\end{equation*}

After evaluating the traces, we have $\mathop\mathrm{Tr}(\bm{D})=\mathop\mathrm{tr}\bm{D}\oplus\mathop\mathrm{tr}(\bm{D}^{2})\oplus\mathop\mathrm{tr}(\bm{D}^{3})=0=\mathbb{1}$, and hence  conclude that the conditions of Theorem~\ref{T-dyyd-Axy-Bxx-Cyx} are fulfilled.

Furthermore, we successively obtain
\begin{gather*}
\bm{D}^{\ast}
=
\bm{I}\oplus\bm{D}\oplus\bm{D}^{2}
=
\left(
\begin{array}{rrc}
0 & -2 & 1
\\
2 & 0 & 3
\\
-1 & -3 & 0
\end{array}
\right),
\qquad
\bm{A}\bm{D}^{\ast}
=
\left(
\begin{array}{crc}
4 & 2 & 5
\\
5 & 3 & 6
\\
1 & -1 & 2
\end{array}
\right),
\\
(\bm{d}^{-}\bm{A}\bm{D}^{\ast})^{-}
=
\left(
\begin{array}{r}
0
\\
2
\\
-1
\end{array}
\right),
\qquad
\bm{D}^{\ast}(\bm{d}^{-}\bm{A}\bm{D}^{\ast})^{-}
=
\left(
\begin{array}{r}
0
\\
2
\\
-1
\end{array}
\right),
\\
\bm{A}\bm{D}^{\ast}(\bm{d}^{-}\bm{A}\bm{D}^{\ast})^{-}
=
\left(
\begin{array}{c}
4
\\
5
\\
1
\end{array}
\right).
\end{gather*}

The minimum value of the objective function, which shows the minimal violation of the due dates, is given by $\Delta=((\bm{A}\bm{D}^{\ast}(\bm{d}^{-}\bm{A}\bm{D}^{\ast})^{-})^{-}\bm{d})^{1/2}=2$.

The optimal schedule has the latest start and finish times defined by the vectors
\begin{equation*}
\bm{x}
=
\Delta\bm{D}^{\ast}(\bm{d}^{-}\bm{A}\bm{D}^{\ast})^{-}
=
\left(
\begin{array}{c}
2
\\
4
\\
1
\end{array}
\right),
\qquad
\bm{y}
=
\Delta\bm{A}\bm{D}^{\ast}(\bm{d}^{-}\bm{A}\bm{D}^{\ast})^{-}
=
\left(
\begin{array}{c}
6
\\
7
\\
3
\end{array}
\right).
\end{equation*}
\end{example}

\subsection{Minimization of maximum deviation of finish times}

After rewriting problem \eqref{P-yiyi-aijxjyi-bijxjxj-cijyjxi-yifi} in terms of the semifield $\mathbb{R}_{\max,+}$, the problem becomes
\begin{equation*}
\begin{aligned}
&
\text{minimize}
&&
\bigoplus_{i=1}^{n}y_{i}
\bigoplus_{j=1}^{n}y_{j}^{-1},
\\
&
\text{subject to}
&&
\bigoplus_{j=1}^{n}a_{ij}x_{j}
=
y_{i},
\quad
\bigoplus_{j=1}^{n}b_{ij}x_{j}
\leq
x_{i},
\quad
\bigoplus_{j=1}^{n}c_{ij}y_{j}
\leq
x_{i},
\\
&&&
y_{i}
\leq
f_{i},
\qquad
i=1,\ldots,n.
\end{aligned}
\end{equation*}

In addition to the previously introduced matrix-vector notation, we define the vector $\bm{f}=(f_{i})$ to write the problem 
\begin{equation}
\begin{aligned}
&
\text{minimize}
&&
\bm{1}^{T}\bm{y}\bm{y}^{-}\bm{1},
\\
&
\text{subject to}
&&
\bm{A}\bm{x}
=
\bm{y},
\quad
\bm{B}\bm{x}
\leq
\bm{x},
\quad
\bm{C}\bm{y}
\leq
\bm{x},
\quad
\bm{y}
\leq
\bm{f}.
\end{aligned}
\label{P-1yy1-Axy-Bxx-Cyx-yf}
\end{equation}

The following result provides a solution to the problem.

\begin{theorem}
\label{T-1yy1-Axy-Bxx-Cyx-yf}
Let $\bm{A}$ be a regular matrix, the matrix $\bm{D}=\bm{B}\oplus\bm{C}\bm{A}$ satisfy the condition $\mathop\mathrm{Tr}(\bm{D})\leq\mathbb{1}$, and $\bm{f}$ be a regular vector. Then, the minimum value in problem \eqref{P-1yy1-Axy-Bxx-Cyx-yf} is equal to
\begin{equation*}
\Delta
=
(\bm{A}\bm{D}^{\ast}(\bm{1}^{T}\bm{A}\bm{D}^{\ast})^{-})^{-}\bm{1},
\end{equation*}
and attained if
\begin{equation*}
\bm{x}
=
\alpha\bm{D}^{\ast}(\bm{1}^{T}\bm{A}\bm{D}^{\ast})^{-},
\quad
\bm{y}
=
\alpha\bm{A}\bm{D}^{\ast}(\bm{1}^{T}\bm{A}\bm{D}^{\ast})^{-},
\quad
\alpha
\leq
(\bm{f}^{-}\bm{A}\bm{D}^{\ast}(\bm{1}^{T}\bm{A}\bm{D}^{\ast})^{-})^{-1}.
\end{equation*}
\end{theorem}
\begin{proof}
As before, we substitute $\bm{y}=\bm{A}\bm{x}$ and combine the first two inequalities into one inequality $\bm{D}\bm{x}\leq\bm{x}$, where $\bm{D}=\bm{B}\oplus\bm{C}\bm{A}$. This inequality is then solved by using Theorem~\ref{T-Ax-x} to give the result $\bm{x}=\bm{D}^{\ast}\bm{u}$, where $\bm{u}$ is a regular vector. 

Furthermore, we write the last inequality constraint as $\bm{A}\bm{D}^{\ast}\bm{u}\leq\bm{f}$, and apply Lemma~\ref{L-Ax-d} to find the solution $\bm{u}\leq(\bm{f}^{-}\bm{A}\bm{D}^{\ast})^{-}$. The problem takes the form
\begin{equation*}
\begin{aligned}
&
\text{minimize}
&&
\bm{1}^{T}\bm{A}\bm{D}^{\ast}\bm{u}(\bm{A}\bm{D}^{\ast}\bm{u})^{-}\bm{1},
\\
&
\text{subject to}
&&
\bm{u}
\leq
(\bm{f}^{-}\bm{A}\bm{D}^{\ast})^{-}.
\end{aligned}
\end{equation*}

First, we remove the constraints and solve the obtained unconstrained problem. By applying Theorem~\ref{T-qBxAxp}, where both matrices $\bm{A}$ and $\bm{B}$ are replaced by $\bm{A}\bm{D}^{\ast}$, and both vectors $\bm{p}$ and $\bm{q}$ by $\bm{1}$, we find the minimum $\Delta=(\bm{A}\bm{D}^{\ast}(\bm{1}^{T}\bm{A}\bm{D}^{\ast})^{-})^{-}\bm{1}$, which is attained at the vector $\bm{u}=\alpha(\bm{1}^{T}\bm{A}\bm{D}^{\ast})^{-}$, where $\alpha>\mathbb{0}$.

To find the values of the parameter $\alpha$, which meet the condition $\bm{u}\leq(\bm{f}^{-}\bm{A}\bm{D}^{\ast})^{-}$, we solve the inequality $\alpha(\bm{1}^{T}\bm{A}\bm{D}^{\ast})^{-}
\leq(\bm{f}^{-}\bm{A}\bm{D}^{\ast})^{-}$. By applying Lemma~\ref{L-Ax-d} with $\alpha$ as the unknown, we have $\alpha\leq(\bm{f}^{-}\bm{A}\bm{D}^{\ast}(\bm{1}^{T}\bm{A}\bm{D}^{\ast})^{-})^{-1}$.

It remains to turn back to the vectors $\bm{x}$ and $\bm{y}$, and thus complete the proof.
\end{proof}

\begin{example}
Suppose that we need to minimize the maximum deviation of finish times in the project from Example~\ref{X-dyyd-Axy-Bxx-Cyx}, where, instead of the due dates, deadline constraints apply, given by the vector
\begin{equation*}
\bm{f}
=
\left(
\begin{array}{c}
6
\\
6
\\
6
\end{array}
\right).
\end{equation*}

We take advantage of intermediate results of the previous example to obtain
\begin{gather*}
(\bm{1}^{T}\bm{A}\bm{D}^{\ast})^{-}
=
\left(
\begin{array}{r}
-5
\\
-3
\\
-6
\end{array}
\right),
\qquad
\bm{D}^{\ast}(\bm{1}^{T}\bm{A}\bm{D}^{\ast})^{-}
=
\left(
\begin{array}{r}
-5
\\
-3
\\
-6
\end{array}
\right),
\\
\bm{A}\bm{D}^{\ast}(\bm{1}^{T}\bm{A}\bm{D}^{\ast})^{-}
=
\left(
\begin{array}{r}
-1
\\
0
\\
-4
\end{array}
\right).
\end{gather*}

According to Theorem~\ref{T-1yy1-Axy-Bxx-Cyx-yf}, the minimum deviation of finish times, which can be achieved in the project, is equal to $\Delta=(\bm{A}\bm{D}^{\ast}(\bm{1}^{T}\bm{A}\bm{D}^{\ast})^{-})^{-}\bm{1}=4$.

The optimal schedule is provided by the vectors
\begin{equation*}
\bm{x}
=
\alpha\bm{D}^{\ast}(\bm{1}^{T}\bm{A}\bm{D}^{\ast})^{-}
=
\alpha
\left(
\begin{array}{r}
-5
\\
-3
\\
-6
\end{array}
\right),
\qquad
\bm{y}
=
\alpha\bm{A}\bm{D}^{\ast}(\bm{1}^{T}\bm{A}\bm{D}^{\ast})^{-}
=
\alpha
\left(
\begin{array}{r}
-1
\\
0
\\
-4
\end{array}
\right),
\end{equation*}
where the condition $\alpha\leq(\bm{f}^{-}\bm{A}\bm{D}^{\ast}(\bm{1}^{T}\bm{A}\bm{D}^{\ast})^{-})^{-1}=6$ must be satisfied.

In terms of standard operations, the elements of the vectors $\bm{x}=(x_{1},x_{2},x_{3})^{T}$ and $\bm{y}=(y_{1},y_{2},y_{3})^{T}$ are written as
\begin{alignat*}{4}
x_{1}
&=
\alpha-5,
&\qquad
x_{2}
&=
\alpha-3,
&\qquad
x_{3}
&=
\alpha-6,
\\
y_{1}
&=
\alpha-1,
&\qquad
y_{2}
&=
\alpha,
&\qquad
y_{3}
&=
\alpha-4,
&\qquad
\alpha
&\leq
6.
\end{alignat*} 
\end{example}

\subsection{Minimization of maximum flow-time}

Consider problem \eqref{P-yixi-aijxjyi-bijxjxi-cijyjxi-gixi}, and rewrite it in terms of the semifield $\mathbb{R}_{\max,+}$. As a result, we obtain the problem
\begin{equation*}
\begin{aligned}
&
\text{minimize}
&&
\bigoplus_{i=1}^{n}x_{i}^{-1}y_{i},
\\
&
\text{subject to}
&&
\bigoplus_{j=1}^{n}a_{ij}x_{j}
=
y_{i},
\quad
\bigoplus_{j=1}^{n}b_{ij}x_{j}
\leq
x_{i},
\quad
\bigoplus_{j=1}^{n}c_{ij}y_{j}
\leq
x_{i},
\\
&&&
g_{i}
\leq
x_{i},
\qquad
i=1,\ldots,n.
\end{aligned}
\end{equation*}

Furthermore, we add the vector $\bm{g}=(g_{i})$. Switching to matrix-vector notation puts the problem in the form
\begin{equation}
\begin{aligned}
&
\text{minimize}
&&
\bm{x}^{-}\bm{y},
\\
&
\text{subject to}
&&
\bm{A}\bm{x}
=
\bm{y},
\quad
\bm{B}\bm{x}
\leq
\bm{x},
\quad
\bm{C}\bm{y}
\leq
\bm{x},
\quad
\bm{g}
\leq
\bm{x}.
\end{aligned}
\label{P-xy-Axy-Bxx-Cyx-gx}
\end{equation}

A complete solution of the problem is given as follows.
\begin{theorem}
\label{T-xAx-BCAxgx}
Let $\bm{A}$ be a matrix with spectral radius $\lambda>\mathbb{0}$, and the matrix $\bm{D}=\bm{B}\oplus\bm{C}\bm{A}$ satisfy the condition $\mathop\mathrm{Tr}(\bm{D})\leq\mathbb{1}$. Then, the minimum value in problem \eqref{P-xy-Axy-Bxx-Cyx-gx} is equal to
\begin{equation*}
\theta
=
\lambda
\oplus
\bigoplus_{k=1}^{n-1}\mathop{\bigoplus\hspace{1.0em}}_{1\leq i_{1}+\cdots+i_{k}\leq n-k}\mathop\mathrm{tr}\nolimits^{1/k}(\bm{A}\bm{D}^{i_{1}}\cdots\bm{A}\bm{D}^{i_{k}}),
\end{equation*}
and all regular solutions are given by
\begin{equation*}
\bm{x}
=
(\theta^{-1}\bm{A}\oplus\bm{D})^{\ast}\bm{u},
\qquad
\bm{y}
=
\bm{A}(\theta^{-1}\bm{A}\oplus\bm{D})^{\ast}\bm{u},
\qquad
\bm{u}
\geq
\bm{g}.
\end{equation*}
\end{theorem}
\begin{proof}
By substitution of the equality constraint $\bm{y}=\bm{A}\bm{x}$, we eliminate the vector $\bm{y}$. Then, we combine all inequality constraints into one to write the problem
\begin{equation*}
\begin{aligned}
&
\text{minimize}
&&
\bm{x}^{-}\bm{A}\bm{x},
\\
&
\text{subject to}
&&
(\bm{B}\oplus\bm{C}\bm{A})\bm{x}\oplus\bm{g}
\leq
\bm{x}.
\end{aligned}
\end{equation*}

This problem has the form of that at \eqref{P-xAx-Bxgx}, where $\bm{B}$ is replaced by $\bm{D}=\bm{B}\oplus\bm{C}\bm{A}$. Thus, a direct application of Theorem~\ref{T-xAx-Bxgx} yields the required solution.
\end{proof}

As an illustration of the solution obtained, we present the next example.

\begin{example}
Consider the problem of minimizing the maximum flow-time in the project, which has the start-finish, start-start and finish-start constraints defined as in Example~\ref{X-dyyd-Axy-Bxx-Cyx}, and release time constraints given by
\begin{equation*}
\bm{g}
=
\left(
\begin{array}{c}
2
\\
2
\\
1
\end{array}
\right).
\end{equation*}

First, we verify the conditions of Theorem~\ref{T-xAx-BCAxgx}. We calculate the matrices
\begin{equation*}
\bm{A}^{2}
=
\left(
\begin{array}{ccr}
8 & 4 & -1
\\
5 & 6 & 2
\\
4 & 1 & 4
\end{array}
\right),
\qquad
\bm{A}^{3}
=
\left(
\begin{array}{ccc}
12 & 8 & 3
\\
9 & 9 & 5
\\
8 & 4 & 6
\end{array}
\right),
\end{equation*}
and then evaluate the traces to find $\lambda=\mathop\mathrm{tr}\bm{A}\oplus\mathop\mathrm{tr}\nolimits^{1/2}(\bm{A}^{2})\oplus\mathop\mathrm{tr}\nolimits^{1/3}(\bm{A}^{3})=
4$. Since $\lambda>0=\mathbb{1}$ and $\mathop\mathrm{Tr}(\bm{D})=0=\mathbb{1}$, the conditions of the theorem are satisfied.

To evaluate the minimum $\theta$ of the objective function, we calculate the matrices
\begin{gather*}
\bm{A}\bm{D}
=
\left(
\begin{array}{crc}
3 & 2 & 5
\\
4 & 2 & 6
\\
1 & -2 & 1
\end{array}
\right),
\qquad
\bm{A}\bm{D}^{2}
=
\left(
\begin{array}{crc}
4 & 1 & 5
\\
5 & 2 & 5
\\
0 & -1 & 2
\end{array}
\right),
\\
\bm{A}\bm{D}\bm{A}
=
\left(
\begin{array}{ccc}
7 & 5 & 7
\\
8 & 5 & 8
\\
5 & 1 & 3
\end{array}
\right),
\qquad
\bm{A}^{2}\bm{D}
=
\left(
\begin{array}{ccc}
7 & 6 & 9
\\
7 & 5 & 9
\\
3 & 2 & 5
\end{array}
\right).
\end{gather*}

After taking the traces and considering $\lambda$, we have the result of minimizing the maximum flow-time, given by $\theta=\lambda\oplus\mathop\mathrm{tr}(\bm{A}\bm{D}\oplus\bm{A}\bm{D}^{2})\oplus\mathop\mathrm{tr}\nolimits^{1/2}(\bm{A}\bm{D}\bm{A}\oplus\bm{A}^{2}\bm{D})=4$.

To describe the solution offered by Theorem~\ref{T-xAx-BCAxgx}, we need the matrices
\begin{equation*}
\theta^{-1}\bm{A}
\oplus
\bm{D}
=
\left(
\begin{array}{rrr}
0 & -2 & 1
\\
1 & -1 & 3
\\
-1 & -6 & -2
\end{array}
\right),
\qquad
(\theta^{-1}\bm{A}
\oplus
\bm{D})^{2}
=
\left(
\begin{array}{rrc}
0 & -2 & 1
\\
2 & -1 & 2
\\
-1 & -3 & 0
\end{array}
\right).
\end{equation*}

By combining these matrices with the identity matrix, we obtain
\begin{equation*}
(\theta^{-1}\bm{A}
\oplus
\bm{D})^{\ast}
=
\left(
\begin{array}{rrc}
0 & -2 & 1
\\
2 & 0 & 3
\\
-1 & -3 & 0
\end{array}
\right).
\end{equation*}

Note that the last matrix can be represented in the form
\begin{equation*}
(\theta^{-1}\bm{A}
\oplus
\bm{D})^{\ast}
=
\left(
\begin{array}{c}
1
\\
3
\\
0
\end{array}
\right)
\left(
\begin{array}{rrc}
-1
&
-3
&
0
\end{array}
\right).
\end{equation*}

The vector of optimal start times provided by Theorem~\ref{T-xAx-BCAxgx} is given by
\begin{equation*}
\bm{x}
=
(\theta^{-1}\bm{A}\oplus\bm{D})^{\ast}\bm{u}
=
\left(
\begin{array}{c}
1
\\
3
\\
0
\end{array}
\right)
\left(
\begin{array}{rrc}
-1
&
-3
&
0
\end{array}
\right)
\bm{u},
\qquad
\bm{u}
\geq
\bm{g}.
\end{equation*}

To simplify the solution, we introduce the new variable $
v
=
\left(
\begin{array}{rrc}
-1
&
-3
&
0
\end{array}
\right)
\bm{u}$,
which has to satisfy the condition
$
v
=
\left(
\begin{array}{rrc}
-1
&
-3
&
0
\end{array}
\right)
\bm{u}
\geq
\left(
\begin{array}{rrc}
-1
&
-3
&
0
\end{array}
\right)
\bm{g}
=
1$.

The vectors of optimal start and finish times now become
\begin{equation*}
\bm{x}
=
\left(
\begin{array}{c}
1
\\
3
\\
0
\end{array}
\right)
v,
\qquad
\bm{y}
=
\bm{A}
\bm{x}
=
\left(
\begin{array}{c}
5
\\
6
\\
2
\end{array}
\right)
v,
\qquad
v
\geq
1.
\end{equation*}

Using standard operations yields the elements of the vectors given by
\begin{alignat*}{4}
x_{1}
&=
v+1,
&\qquad
x_{2}
&=
v+3,
&\qquad
x_{3}
&=
v,
\\
y_{1}
&=
v+5,
&\qquad
y_{2}
&=
v+6,
&\qquad
y_{3}
&=
v+2,
&\qquad
v
&\geq
1.
\end{alignat*} 
\end{example}

\subsection{Minimization of makespan}

In the framework of the semifield $\mathbb{R}_{\max,+}$, problem \eqref{P-yixi-aijxjyi-gixihi-yifi} of minimizing the makespan is rewritten as
\begin{equation*}
\begin{aligned}
&
\text{minimize}
&&
\bigoplus_{i=1}^{n}y_{i}
\bigoplus_{j=1}^{n}x_{j}^{-1},
\\
&
\text{subject to}
&&
\bigoplus_{j=1}^{n}a_{ij}x_{j}
=
y_{i},
\quad
g_{i}
\leq
x_{i}
\leq
h_{i},
\\
&&&
y_{i}
\leq
f_{i},
\qquad
i=1,\ldots,n.
\end{aligned}
\end{equation*}

By adding the vector $\bm{h}=(h_{i})$ and using $\bm{1}$ to indicate the vector of ones, we represent the problem as follows:
\begin{equation}
\begin{aligned}
&
\text{minimize}
&&
\bm{1}^{T}\bm{y}\bm{x}^{-}\bm{1},
\\
&
\text{subject to}
&&
\bm{A}\bm{x}
=
\bm{y},
\quad
\bm{g}
\leq
\bm{x}
\leq
\bm{h},
\quad
\bm{y}
\leq
\bm{f}.
\end{aligned}
\label{P-1yx1-Axy-gxh-yf}
\end{equation}

\begin{theorem}
\label{T-1yx1-Axy-gxh-yf}
Let $\bm{A}$ be a column-regular matrix, $\bm{h}$ and $\bm{f}$ be regular vectors such that $(\bm{h}^{-}\oplus\bm{f}^{-}\bm{A})\bm{g}\leq\mathbb{1}$. Then, the minimum value in problem \eqref{P-1yx1-Axy-gxh-yf} is equal to
\begin{equation*}
\theta
=
\bm{1}^{T}\bm{A}(\bm{I}\oplus\bm{g}\bm{h}^{-})\bm{1},
\end{equation*}
and all regular solutions are given by
\begin{equation*}
\bm{x}
=
(\bm{I}\oplus\theta^{-1}\bm{1}\bm{1}^{T}\bm{A})\bm{u},
\qquad
\bm{y}
=
\bm{A}(\bm{I}\oplus\theta^{-1}\bm{1}\bm{1}^{T}\bm{A})\bm{u},
\end{equation*}
where
\begin{equation*}
\bm{g}
\leq
\bm{u}
\leq
((\bm{h}^{-}\oplus\bm{f}^{-}\bm{A})(\bm{I}\oplus\theta^{-1}\bm{1}\bm{1}^{T}\bm{A}))^{-}.
\end{equation*}
\end{theorem}
\begin{proof}
As before, we first substitute $\bm{y}=\bm{A}\bm{x}$. Solving the inequality $\bm{A}\bm{x}\leq\bm{f}$ by Lemma~\ref{L-Ax-d} yields $\bm{x}\leq(\bm{f}^{-}\bm{A})^{-}$. Then, we take the upper boundaries $\bm{x}\leq\bm{h}$ and $\bm{x}\leq(\bm{f}^{-}\bm{A})^{-}$, and apply conjugate transposition to rewrite the inequalities as $\bm{x}^{-}\geq\bm{h}^{-}$ and $\bm{x}^{-}\geq\bm{f}^{-}\bm{A}$. By coupling both inequalities into one, and again taking the conjugate transposition, we obtain one upper bound $\bm{x}\leq(\bm{h}^{-}\oplus\bm{f}^{-}\bm{A})^{-}$.

We write the objective function as $\bm{1}^{T}\bm{A}\bm{x}\bm{x}^{-}\bm{1}=\bm{x}^{-}\bm{1}\bm{1}^{T}\bm{A}\bm{x}$ to obtain the problem
\begin{equation}
\begin{aligned}
&
\text{minimize}
&&
\bm{x}^{-}\bm{1}\bm{1}^{T}\bm{A}\bm{x},
\\
&
\text{subject to}
&&
\bm{g}
\leq
\bm{x}
\leq
(\bm{h}^{-}\oplus\bm{f}^{-}\bm{A})^{-}.
\end{aligned}
\end{equation}

The problem obtained is of the form of \eqref{P-xAx-gxh}, where $\bm{A}$ is replaced by $\bm{1}\bm{1}^{T}\bm{A}$ and $\bm{h}$ by $(\bm{h}^{-}\oplus\bm{f}^{-}\bm{A})^{-}$. To apply Theorem~\ref{T-xAx-gxh}, we need to find the spectral radius of the matrix $\bm{1}\bm{1}^{T}\bm{A}$. We first calculate, for each $k=1,\ldots,n$,
\begin{equation*}
(\bm{1}\bm{1}^{T}\bm{A})^{k}
=
(\bm{1}^{T}\bm{A}\bm{1})^{k-1}\bm{1}\bm{1}^{T}\bm{A},
\qquad
\mathop\mathrm{tr}(\bm{1}\bm{1}^{T}\bm{A})^{k}
=
(\bm{1}^{T}\bm{A}\bm{1})^{k},
\end{equation*}
from which it follows that the spectral radius is equal to $\lambda=\bm{1}^{T}\bm{A}\bm{1}>\mathbb{0}$.

Furthermore, we consider the minimum given by 
\begin{equation*}
\theta
=
\bm{1}^{T}\bm{A}\bm{1}
\oplus
(\bm{1}^{T}\bm{A}\bm{1})\bigoplus_{k=1}^{n-1}((\bm{1}^{T}\bm{A}\bm{1})^{-1}\bm{h}^{-}\bm{1}\bm{1}^{T}\bm{A}\bm{g})^{1/k}.
\end{equation*}

Suppose that $\bm{1}^{T}\bm{A}\bm{1}\leq\bm{h}^{-}\bm{1}\bm{1}^{T}\bm{A}\bm{g}$. Since the inequality $(\bm{1}^{T}\bm{A}\bm{1})^{-1}\bm{h}^{-}\bm{1}\bm{1}^{T}\bm{A}\bm{g}\geq\mathbb{1}$ holds, we have $((\bm{1}^{T}\bm{A}\bm{1})^{-1}\bm{h}^{-}\bm{1}\bm{1}^{T}\bm{A}\bm{g})^{1/k}\leq(\bm{1}^{T}\bm{A}\bm{1})^{-1}\bm{h}^{-}\bm{1}\bm{1}^{T}\bm{A}\bm{g}$, and, therefore, conclude that $\theta=\bm{h}^{-}\bm{1}\bm{1}^{T}\bm{A}\bm{g}$.

On the other hand, if $\bm{1}^{T}\bm{A}\bm{1}>\bm{h}^{-}\bm{1}\bm{1}^{T}\bm{A}\bm{g}$, then we have $\theta=\bm{1}^{T}\bm{A}\bm{1}$. By combining both results, we finally obtain $\theta=\bm{1}^{T}\bm{A}\bm{1}\oplus\bm{h}^{-}\bm{1}\bm{1}^{T}\bm{A}\bm{g}=\bm{1}^{T}\bm{A}(\bm{I}\oplus\bm{g}\bm{h}^{-})\bm{1}$.

To describe the solution set according to Theorem~\ref{T-xAx-gxh}, we examine the matrix
\begin{equation*}
(\theta^{-1}\bm{1}\bm{1}^{T}\bm{A})^{\ast}
=
\bigoplus_{k=0}^{n-1}(\theta^{-1}\bm{1}\bm{1}^{T}\bm{A})^{k}
=
\bm{I}
\oplus
\theta^{-1}\bigoplus_{k=1}^{n-1}(\theta^{-1}\bm{1}^{T}\bm{A}\bm{1})^{k-1}\bm{1}\bm{1}^{T}\bm{A}.
\end{equation*}

Considering the inequality $\theta\geq\bm{1}^{T}\bm{A}\bm{1}$, we have $(\theta^{-1}\bm{1}\bm{1}^{T}\bm{A})^{\ast}=\bm{I}\oplus\theta^{-1}\bm{1}\bm{1}^{T}\bm{A}$. Substitution into the solution provided by Theorem~\ref{T-xAx-gxh} yields $\bm{x}=(\bm{I}\oplus\theta^{-1}\bm{1}\bm{1}^{T}\bm{A})\bm{u}$, where the vector $\bm{u}$ satisfies the condition $\bm{g}\leq\bm{u}\leq((\bm{h}^{-}\oplus\bm{f}^{-}\bm{A})(\bm{I}\oplus\theta^{-1}\bm{1}\bm{1}^{T}\bm{A}))^{-}$.

Finally, we represent the vector $\bm{y}=\bm{A}\bm{x}$, which completes the proof.
\end{proof}

\begin{example}
Assume that we need to find a schedule with the minimum makespan under the start-finish, release time and deadline constraints, which are defined by the matrix $\bm{A}$ and the vectors $\bm{g}$ and $\bm{f}$ in the previous examples. Suppose that, in addition, release deadlines have to be taken into account, given by the vector
\begin{equation*}
\bm{h}
=
\left(
\begin{array}{c}
3
\\
3
\\
2
\end{array}
\right).
\end{equation*}
\end{example}

To verify that the condition of Theorem~\ref{T-1yx1-Axy-gxh-yf} is fulfilled, we calculate the vectors 
\begin{equation*}
\bm{f}^{-}\bm{A}
=
\left(
\begin{array}{rrr}
-2
&
-3
&
-4
\end{array}
\right),
\qquad
\bm{h}^{-}\oplus\bm{f}^{-}\bm{A}
=
\left(
\begin{array}{rrr}
-2
&
-3
&
-2
\end{array}
\right),
\end{equation*}
and then obtain the required condition in the form $(\bm{h}^{-}\oplus\bm{f}^{-}\bm{A})\bm{g}=0=\mathbb{1}$.

We now apply Theorem~\ref{T-1yx1-Axy-gxh-yf} to find the optimal schedule. The calculation of the minimum makespan $\theta$ involves
\begin{equation*}
\bm{I}
\oplus
\bm{g}\bm{h}^{-}
=
\left(
\begin{array}{rrc}
0 & -1 & 0
\\
-1 & 0 & 0
\\
-2 & -2 & 0
\end{array}
\right),
\qquad
(\bm{I}\oplus\bm{g}\bm{h}^{-})\bm{1}
=
\left(
\begin{array}{c}
0
\\
0
\\
0
\end{array}
\right),
\\
\bm{1}^{T}\bm{A}
=
\left(
\begin{array}{ccc}
4
&
3
&
2
\end{array}
\right),
\end{equation*}
which results in $\theta=\bm{1}^{T}\bm{A}(\bm{I}\oplus\bm{g}\bm{h}^{-})\bm{1}=4$.

To represent the solution, we need to find the matrices
\begin{equation*}
\bm{1}\bm{1}^{T}\bm{A}
=
\left(
\begin{array}{ccc}
4 & 3 & 2
\\
4 & 3 & 2
\\
4 & 3 & 2
\end{array}
\right),
\qquad
\bm{I}\oplus\theta^{-1}\bm{1}\bm{1}^{T}\bm{A}
=
\left(
\begin{array}{crr}
0 & -1 & -2
\\
0 & 0 & -2
\\
0 & -1 & 0
\end{array}
\right),
\end{equation*}
and to form the vectors
\begin{equation*}
\bm{u}_{1}
=
\bm{g}
=
\left(
\begin{array}{c}
2
\\
2
\\
1
\end{array}
\right),
\qquad
\bm{u}_{2}
=
((\bm{h}^{-}\oplus\bm{f}^{-}\bm{A})(\bm{I}\oplus\theta^{-1}\bm{1}\bm{1}^{T}\bm{A}))^{-}
=
\left(
\begin{array}{c}
2
\\
3
\\
2
\end{array}
\right).
\end{equation*}

The vectors of the optimal start and finish times in the project take the form
\begin{equation*}
\bm{x}
=
(\bm{I}\oplus\theta^{-1}\bm{1}\bm{1}^{T}\bm{A})\bm{u}
=
\left(
\begin{array}{crr}
0 & -1 & -2
\\
0 & 0 & -2
\\
0 & -1 & 0
\end{array}
\right)
\bm{u},
\qquad
\bm{y}
=
\bm{A}\bm{x}
=
\left(
\begin{array}{ccc}
4 & 3 & 2
\\
3 & 3 & 1
\\
2 & 1 & 2
\end{array}
\right)
\bm{u},
\end{equation*}
where the vector $\bm{u}$ satisfies the condition $\bm{u}_{1}\leq\bm{u}\leq\bm{u}_{2}$.

Note that the solution can be simplified as follows. The lower and upper bounds $\bm{u}_{1}$ and $\bm{u}_{2}$ for the vector $\bm{u}$ yield the corresponding bounds on $\bm{x}$ in the form
\begin{equation*}
\bm{x}_{1}
=
(\bm{I}\oplus\theta^{-1}\bm{1}\bm{1}^{T}\bm{A})\bm{u}_{1}
=
\left(
\begin{array}{c}
2
\\
2
\\
2
\end{array}
\right),
\qquad
\bm{x}_{2}
=
(\bm{I}\oplus\theta^{-1}\bm{1}\bm{1}^{T}\bm{A})\bm{u}_{2}
=
\left(
\begin{array}{c}
2
\\
3
\\
2
\end{array}
\right).
\end{equation*}

Since the first and third elements of $\bm{x}_{1}$ and $\bm{x}_{2}$ coincide, the elements of the vector $\bm{x}$ can be directly defined by $x_{1}=2$, $2\leq x_{2}\leq3$ and $x_{3}=2$.

After calculating the vector $\bm{y}=\bm{A}\bm{x}$, we write the results using a parameter $v$ and standard operations as follows:
\begin{alignat*}{4}
x_{1}
&=
2,
&\qquad
x_{2}
&=
v,
&\qquad
x_{3}
&=
2,
\\
y_{1}
&=
6,
&\qquad
y_{2}
&=
v+3,
&\qquad
y_{3}
&=
4,
&\qquad
2
\leq
v
\leq
3.
\end{alignat*}

\section{Conclusions}

In this paper, we demonstrated new solution techniques, based on the models and methods of tropical mathematics, for a class of project scheduling problems with minimax objectives. The paper extends and generalizes results in \cite{Krivulin2012Algebraic,Krivulin2014Aconstrained,Krivulin2015Amultidimensional,Krivulin2015Extremal,Krivulin2015Tropical} by solving problems not previously considered, which involve different objective functions and/or more complicated systems of constraints.

It was shown that many typical constraints and objectives that occur in time-constrained project scheduling are naturally represented in terms of tropical algebra in a compact vector form. Specifically, the start-finish, start-start and finish-start precedence relationships are readily given by linear vector equations and inequalities, whereas the project makespan, the maximum deviation from due dates and the maximum flow-time optimization criteria can be well written as nonlinear functions defined on vectors through a multiplicative conjugate transposition operator. 

The scheduling problems of interest were reduced to tropical optimization problems, which are solved through the application and further development of recent results in tropical optimization. The solutions are given in direct, explicit vector forms ready for formal analysis and straightforward computation with low polynomial complexity. The explicit form of the solutions can be considered as a definite advantage over the algorithmic techniques normally used in scheduling. Specifically, the problems under study can be formulated as linear programs, which offers algorithmic solutions using one of the iterative computational schemes of linear programming, but does not guarantee a direct closed-form solution.

The contribution of the paper is twofold. First, we have developed new applications of tropical optimization, formulated new optimization problems, extended existing solution methods to these problems, and derived direct solutions. Second, we have devised a new solution approach for time-constrained scheduling problems, which involves the representation of the problems in the tropical mathematics setting, and the application of methods from tropical optimization to obtain direct solutions in a closed form. Analytical techniques proposed and described in the paper can serve as a template for the solution of other optimization problems, and for the application of the solutions to real-world problems in various fields.

Possible lines of further investigation include the extension of the approach to account for new types of constraints and objectives in the scheduling problems under study, and to solve new classes of problems, including scheduling problems with renewable resources. Computational experiments with real-world data present another research topic of interests. 

\bibliographystyle{abbrvurl}
\bibliography{Tropical_optimization_problems_in_time_constrained_project_scheduling}

\end{document}